\documentclass{amsart}
\usepackage{graphicx,amsthm,amssymb,amsmath,mathtools}
\usepackage{hyperref}
\usepackage{color}

\theoremstyle{plain}
\newtheorem{theorem}{Theorem}[section]
\newtheorem{proposition}[theorem]{Proposition}

\theoremstyle{definition}
\newtheorem{definition}[theorem]{Definition}
\newtheorem{example}[theorem]{Example}

\theoremstyle{remark}
\newtheorem{remark}[theorem]{Remark}

\newcommand{\highlight}[1]{\emph{\textcolor{blue}{#1}}}

\numberwithin{equation}{section}

\newcommand{\bZ}{\mathbb{Z}}

\newcommand{\Spots}{\mathcal{P}}
\newcommand{\ls}{\mathrm{ls}}
\newcommand{\Tspots}{\widehat{\mathcal{P}}}
\newcommand{\val}{\mathrm{val}}

\newcommand{\Park}{\mathrm{Park}}
\newcommand{\fin}{\mathrm{Fin}}
\newcommand{\dir}{\mathrm{Dir}}

\title{Bilateral parking procedures
}

\author{Philippe Nadeau}
\address{Univ Lyon, CNRS, Université Claude Bernard Lyon 1, UMR 5208, Institut Camille Jordan, 43 Blvd. du 11 Novembre 1918, F-69622 Villeurbanne Cedex, France}
\email{\href{mailto:nadeau@math.univ-lyon1.fr}{nadeau@math.univ-lyon1.fr}}
\thanks{This research is partially supported by the French project ANR19-CE48-011-01 (COMBINÉ)}

\begin{document}

\begin{abstract} We introduce the class of { bilateral}  parking procedures on the discrete line $\bZ$. While cars try to park in the nearest available spot to their right in the classical case, we consider more general parking rules that allow cars to use the nearest available spot to their left. We show that for a natural subclass of local procedures, the number of corresponding { parking functions} of length $r$ is always equal to $(r+1)^{r-1}$. The setting can be extended to probabilistic procedures, in which the decision to park left or right is random. We finally describe how bilateral procedures can naturally be encoded by certain labeled binary forests, whose combinatorics shed light on several results from the literature.
\end{abstract}

\maketitle

%%%%%%%%%%%%%%%%%%%%%%%%%%%%%
\section{Introduction}
\label{sec:introduction}
%%%%%%%%%%%%%%%%%%%%%%%%%%%%%

Classical parking functions are a central object in enumerative and algebraic combinatorics. They are connected to various structures such as noncrossing partitions, hyperplane arrangements, and many others: see for instance the survey~\cite{yan2015parking} and references therein. The corresponding parking procedure $\Spots^{right}$ on $\bZ$ was originally defined as an elementary hashing procedure, cf.~\cite{konheim1966occupancy}.
\smallskip

We recall its definition:  $r$ cars want to park on an empty one-way street with spots labeled by $1,2,\ldots,r$ from left to right. The cars arrive successively, and the $i$th car has a preferred spot $a_i$. If this spot is available, it parks there, and if not it parks in the \emph{nearest available spot to the right}. The sequence $(a_1,\ldots, a_r)$ is called a \highlight{parking function} if at the end, all cars managed to park. The number of parking functions for $r$ cars is given by the simple formula $(r+1)^{r-1}$~\cite{konheim1966occupancy}. 

The following characterization is well known: $(a_1,\ldots, a_r)$ is a parking function if and only if for any $k=1,\ldots,r$, there are at least $k$ indices $i$ such that $1\leq a_i\leq k$. Parking function generalizations have often relied on this description. This is the case of $G$-parking functions and $\mathbf{u}$-parking functions, see \cite{yan2015parking}. Recently, several variations were also considered by varying different aspects of the procedure, see for instance \cite{carlson2021adventure} and references therein.\smallskip

In this work, we will consider the extensions of the parking procedure obtained by simply changing the ``nearest available spot to the right'' condition: we will also allow one to park to the nearest available spot \emph{to the left}. This explains the name bilateral for our class of procedures. 

Let us give already two examples of possible rules. We need only describe where to park when one's preferred spot is occupied:

\begin{itemize}
\item \highlight{$\Spots^{closest}$}: If the nearest available spot to the right is (weakly) closer to $a_i$ than the nearest available spot to the left, park there; otherwise park to the left.
\item \highlight{$\Spots^{prime}$}: If the total number of cars parked between the nearest available spots to the left and to the right is a prime number, park to the right; otherwise park to the left.
\end{itemize}

In general we will describe parking procedures $\Spots$ as functions associating a finite subset of $\bZ$ of cardinality $r$ to a \textit{preference word} $a_1\cdots a_r$, so that $\Spots(a_1\cdots a_r)$ is the set of occupied spots after $r$ cars have parked. If it is equal to $\{1,\ldots,r\}$ then we will say that $a_1\cdots a_r$ is a $\Spots-$parking function.
The notion of bilateral procedure is then easy to encode as conditions on the function $\Spots$. 

We will then state two conditions to determine the subclass of \highlight{local} procedures: roughly put, these say that left/right decisions must be invariant under translation, and depend only on the cars that parked around the desired spot. We then obtain the following striking enumerative result.

\begin{theorem}
\label{theorem:universal}
Let $\Spots$ be a bilateral, local parking procedure. Then the number of $\Spots$-parking functions of length $r$ is given by $(r+1)^{r-1}$.
\end{theorem}

In particular this holds for the classical procedure denoted $\Spots^{right}$, for $\Spots^{closest}$ and $\Spots^{prime}$ defined above (and for uncountably many other procedures). 
One may interpret Theorem~\ref{theorem:universal} as a kind of discrete \emph{universality} result: a unique formula holds for a large classe of procedures where only local conditions are imposed. %The terminology is borrowed from statistical physics. 
The proof of this result will follow from Pollak's argument in the classical case, based on a cyclic procedure derived from the original one.

We can naturally add randomness to a bilateral procedure: instead of choosing either left or right, pick probabilities for the two options. Each preference word now has a certain probability to be $\Spots$-parking. For a particular procedure, we will see that these probabilities, suitably normalized, coincide with the family of remixed Eulerian numbers studied by Vasu Tewari and the author~\cite{nt2023remixed}. The notion of \emph{local} procedure extends to the probabilistic setting, and a version of Theorem~\ref{theorem:universal} holds in this case.

From the point of view of bijective combinatorics, we show that bilateral procedures are naturally encoded by certain pairs of labeled forests. This bijection naturally lifts the original parking procedure, and is directly connected to the \textit{outcome function} of the procedure, which encodes the order in which the cars parked. This encoding recovers some known results, including in the classical case.\medskip

\noindent{\bf Outline. }We first describe rigorously the parking procedures $\Spots$  that we consider in Section~\ref{sec:bilateral}. We will focus on bilateral procedures, and then introduce the subclass of local procedures in Section~\ref{sec:local}. We will then see that the enumeration $(r+1)^{r-1}$ is in a sense universal for local procedures, see Theorem \ref{theorem:universal}. We then explain how to define a probabilistic version of our procedures in Section~\ref{sec:probabilistic}. We describe some natural connection with the combinatorics of binary trees via a natural encoding in Section~\ref{sec:binary_trees}. We finally briefly define a ``colored'' version of the model in the last section.

%%%%%%%%%%%%%%%%%%%%%%%%%%%%%
\section{Bilateral and local parking procedures}
\label{sec:bilateral}
%%%%%%%%%%%%%%%%%%%%%%%%%%%%%

 As stated in the introduction, the list of parking spot preferences of incoming cars is given by a preference word $a_1a_2\cdots a_r$ with $a_i\in\bZ$: the $i$th car wants to park in the spot $a_i$. A general parking procedure $\Spots$ goes as follows: it is defined inductively, with no cars parked at the beginning. Assume the first $i-1$ cars have parked at distinct spots in $\bZ$. If the spot $a_i$  is available, the $i$th car parks there. Otherwise, there is a rule that determines an available parking spot for the car $a_i$. In the classical parking procedure, this rule consists in choosing systematically the nearest available spot to the right. Bilateral procedures also allow the spot to be the nearest available spot to the left.

\begin{remark}
It is more customary to define parking procedures on an interval $\{1,\ldots,r\}$, and say that the procedure fails as soon as a car cannot park. A parking function is then an entry word of length $r$ where everyone managed to park. Here we define our procedures on $\bZ$, which is more pleasing from a mathematical point of view since we don't have to deal with partially defined functions, as all cars will find an available spot. A procedure is successful if the occupied spots at the end are $\{1,2,\ldots,r\}$ (see Definition~\ref{definition:parking_function}), which is how the boundary conditions are integrated to the setting.
\end{remark}

\subsection{Parking procedures} Let $\fin(\bZ)=\{I\subset\bZ~|~\#I<+\infty\}$ be the collection of finite subsets of $\bZ$. A parking procedure $\Spots$ is determined by the subsets $\Spots(a_1\cdots a_r)$ representing the set of occupied spots after cars with preferences $a_1,\ldots,a_r\in \bZ$ have parked successively. 

This leads to the following definition:

\begin{definition}[\highlight{Parking procedure}]
\label{definition:parking_procedure}
A function $\Spots:\bZ^*\to \fin(\bZ)$ is a parking procedure if $\Spots(\epsilon)=\emptyset$ and for any $r\geq 1$ and any word $a_1\cdots a_r$,
\begin{enumerate} 
\item The subset $\Spots(a_1\cdots a_r)$ has cardinality $r$, and for any $i<r$ \[\Spots (a_1\cdots a_i)\subset\Spots(a_1\cdots a_{i+1});\]
\item If $a_r\notin\Spots(a_1\cdots a_{r-1})$, then $\Spots(a_1\cdots a_{r})=\Spots(a_1\cdots a_{r-1})\sqcup \{a_r\}$
\end{enumerate}
\end{definition}

The first condition indicates that everyone manages to park (and stays put afterwards), and the second one expresses that one parks at their preferred spot if it is available. Note that at this stage cars may possibly park at any available spot, possibly very far, when their desired spot is not available. We will soon restrict to bilateral procedures.

\begin{remark} These conditions capture reasonable conditions of what a ``real-life'' parking procedure on $\bZ$ should satisfy. Note however that in a recent work~\cite{harris2022outcome}, cars can be dislodged from their parking spot by a later car, which is not covered by our setup.
\end{remark}

\subsection{Memoryless parking procedures}
Definition~\ref{definition:parking_procedure} includes the possibility that the parking decisions may depend on the whole sequence of preferred spots $a_1,\ldots,a_i$. A natural subclass of parking procedures, which includes the classical one $\Spots^{right}$ as well as $\Spots^{prime},\Spots^{closest}$ from the introduction, consists of those where the parking decisions only depend on the set of occupied spots. This can be expressed as follows:

\begin{definition}
A parking procedure $\Spots$ is \highlight{memoryless} if there exists a function $M_{\Spots}:\fin(\bZ)\times \bZ\to\bZ$ such that $\Spots(a_1\cdots a_ra)=\Spots(a_1\cdots a_r)\sqcup \{M_\Spots(\Spots(a_1\cdots a_r),a)\}$ for any  $a_1,\cdots,a_r,a$.
\end{definition}

The function $M_\Spots$  characterizes the parking procedure $\Spots$ by immediate induction. By definition of our procedures, it necessarily satisfies $M(S,a)\notin S$ for any $S,a$, and $M(S,a)=S\sqcup\{a\}$ whenever $a\notin S$. Conversely, any family of integers $M(S,a)\in\bZ$ for $a,S$ such that $a\in S$, that satisfy $M(S,a)\notin S$ defines a memoryless parking procedure. $M(S,a)$ specifies the spot where a car with preferred spot $a$ parks, when $a$ belongs to the set $S$ of occupied spots.\smallskip

Let us now give an example of a parking procedure that is \emph{not} memoryless. 
\begin{example}
The procedure \highlight{$\Spots^{LBS}$} is introduced and studied by Vasu Tewari and the author in~\cite{nt2024forest} under the name $\Omega$. It is defined inductively as follows: Let $W=a_1\cdots a_r$ be any word and let $a\in \Spots^{LBS}(W)$. Let $I$ be the maximal interval of $\Spots^{LBS}(W)$ such that $a\in I$. Define $j\in\{1,\ldots,r\}$ to be the largest integer such that $a_j\in I$: in words, the $j$th car is the last car that parked on $I$. Then $\Spots^{LBS}(Wa)$ is defined by parking in the nearest spot available left of $I$ if $a<a_j$, and right of $I$ if $a\geq a_j$. 
 
  Note that one needs to remember, for each maximal interval, where the last car that parked there had wanted to park: the procedure is thus not memoryless. To give an explicit example, it sends both words $12$ and $21$ to the set $S=\{1,2\}$ (this is true for any procedure of course as preferred spots are available). But it sends $121$ to $\{0,1,2\}$ and $211$ to $\{1,2,3\}$.
\end{example}

The reader will have no trouble coming up with other, simpler examples. We included this one since it was the main inspiration for this work.

\subsubsection{$\Spots$-parking functions} We finally come to the definition of parking functions associated to a parking procedure.

\begin{definition}%[\highlight{$\Spots$-parking functions}]
\label{definition:parking_function}
 Let $\Spots$ be a parking procedure. A word $a_1\cdots a_r$ is parking for $\Spots$, or a {\highlight{$\Spots$-parking function}}, if $\Spots(a_1\cdots a_r)=\{1,\ldots,r\}$. 
\end{definition}

It is clear that this coincides with the usual notion of parking functions when $\Spots=\Spots^{right}$. Let \highlight{$\Park(\Spots)$} be the set of parking functions for $\Spots$, and \highlight{$\Park_r(\Spots)$} those of length $r$. The standard enumeration result recalled in the introduction is $\#\Park_r(\Spots^{right})=(r+1)^{r-1}$.% We extend this to a large class of parking procedures.

An immediate remark is that for any $\Spots$-parking procedure, all $r!$ permutations of the word $12\cdots r$ are in $\Park_r(\Spots)$: indeed these are the cases where everyone finds their preferred spots available. 

\subsubsection{Functions $\ls_\Spots$ and $\pi_\Spots(W)$} Let us define two auxiliary functions associated to a parking procedure $\Spots$.

The function \highlight{$\ls_\Spots$} from $\bZ^*$ to $\bZ$  associates to a preference word the spot in $\bZ$ where the last car parks. Explicitly, $\ls_\Spots(a_1\cdots a_i)=k$ where $\{k\}=\Spots (a_1\cdots a_i)\setminus\Spots(a_1\cdots a_{i-1})$.

 If $W=a_1\cdots a_r$ is any word, then the \highlight{outcome function $\pi_\Spots(W)$} is defined as follows:
 \[\pi_\Spots(W):\Spots(W)\to \mathbb{Z},\quad i\mapsto \ls(a_1a_2\cdots a_i).\] 
 It is the injection that associates to an occupied spot $k$ the integer $i$ such that the $i$th car parked in $k$. When $W$ is $\Spots$-parking, we have that $\pi_\Spots(W)$ is a permutation of $\{1,\ldots,r\}$.
 
 A natural question, well studied in the case of the classical parking procedure, is to compute, given a permutation $\sigma$, the number of (parking) words $W$ such that $\pi_\Spots(W)=\sigma$. This naturally partitions parking words of length $r$ in $r!$ classes. We will come back to this question in Section~\ref{sec:binary_trees}.
 
%%%%%%%%%%%%%%%%%%%%%%
\subsection{Bilateral parking procedures} 
\label{sub:bilateral}

 In this work, we will focus on {\em bilateral} parking procedures: when one's preferred parking spot is not available, then the chosen spot must be either the nearest available spot to the left or to the right\footnote{Even though our goal is not to define real-life parking procedures, the constraint is fairly natural in order for cars to travel a small extra distance.}. Note that all procedures considered up to this point are bilateral.
 
  Let us formulate this in mathematical language, and fix some terminology and notations in the process. Given a subset $S$ of $\bZ$, we say that a discrete interval $I=\{t,t+1,\ldots,u\}$ with $t\leq u$ in $\bZ$ is a \highlight{block} of $S$ if it is included in $S$, and maximal for this property with respect to inclusion. For instance, if $S=\{2,3,5,6,7,8,9,12\}$, its blocks are $\{2,3\}, \{5,6,7,8,9\}$ and $\{12\}$.
 
\begin{definition}[\highlight{Bilateral parking procedure}]
A parking procedure $\Spots$ is \emph{bilateral} if the following holds: For any word $W$ and letter $a$ such that $a\in \Spots(W)$, let $I=\{t,t+1,\ldots,u\}$ be the block of $\Spots(W)$ such that $a\in I$, one has $\ls_\Spots(Wa)\in\{t-1,u+1\}$.
\end{definition} 
  
 Thus to define a bilateral procedure $\Spots$, it suffices to determine a rule that picks either the spot \emph{left} of the block ($\ls_\Spots(Wa)=t-1$) or \emph{right} ( $\ls(Wa)=u+1$)  whenever $a\in \Spots(W)$. We define the direction \highlight{$\dir_\Spots$} by $\dir_\Spots(Wa)\in\{left,right\}$ accordingly, and will more conveniently define bilateral procedures by specifying $\dir_\Spots$. For memoryless procedures, we will naturally abuse notation and write $\dir_\Spots(S,a)$ for $\dir_\Spots(Wa)$ when $S=\Spots(W)$.

\begin{example}
\label{example:kNaples}
Let $k\in \bZ_{>0}$. The \highlight{$k$-Naples parking procedure} {$\Spots_+^{Nap,k}$} is introduced in \cite{christensen2020naples} and further studied in \cite{colmenarejo2021countingknaples}. The procedure is defined on positive integers, i.e. goes from $(\bZ_{>0})^*$ to $\fin(\bZ_{>0})$. It is a variation of the classical parking procedure where one allows cars to back up, up to $k$ spots, to find an available spot.

We give an extension \highlight{$\Spots^{Nap,k}$} to $\bZ^*$ here --note that in order to define parking functions, only the original procedure plays a role, so our notion coincides with the one from \cite{christensen2020naples},\cite{colmenarejo2021countingknaples}. When a driver finds its preferred spot $a$ occupied, it goes to the next available spot to the right if $a\leq 0$. If $a>0$, it checks spots $a-1,a-2,\ldots a-k$ in order and parks to the first available one, say $a-j$, if it exists and satisfies $a-j>0$; otherwise it takes the first available spot to the right. 

Using our notations, $\Spots^{Nap,k}$ is the memoryless procedure characterized as follows: for $S\in\fin(\bZ)$ and $a\in S$, let $I=\{t,t+1,\ldots,u\}$ be the block that contains $a$. Then $\dir(S,a)=left$ if $t>1$ and $a-t<k$, and $\dir(S,a)=right$ otherwise.
\end{example}
 
%%%%%%%%%%%%%%%%%%%%%%%%%%%%%%%%%%%%%%
\section{Local parking procedures}
\label{sec:local}
%%%%%%%%%%%%%%%%%%%%%%%%%%%%%%%%%%%%%%

\subsection{Shift invariance and local decision}
\label{sub:twoproperties}

We now come to two extra constraints on a procedure $\Spots$. Informally, we require it to be invariant under translation, and that left/right decisions must depend only on the subsequence of cars that parked on the block.

\subsubsection{Shift invariance} Let \highlight{$\tau:i\mapsto i+1$} be the shift on $\bZ$. It extends to subsets of $\bZ$ or words in $\bZ^*$ naturally: for instance $\tau(\{2,3,5\})=\{3,4,6\}$ and $\tau(523)=634$.

\begin{definition}[\highlight{Shift invariance}]
\label{definition:shift_invariance}
A bilateral parking procedure $\Spots$ is said to be \emph{shift invariant} if for any word $W$, one has $\dir_\Spots(\tau(W))=\dir_\Spots(W)$. 
\end{definition}

This condition can be equivalently written as $\Spots(\tau(W))=\tau(\Spots(W))$, which allows this definition to be extended to non necessarily bilateral procedures.

\begin{example}
The procedures $\Spots^{right},\Spots^{prime},\Spots^{closest}, \Spots^{LBS}$ are all shift invariant. The $k$-Naples procedure $\Spots^{Nap,k}$ for $k\geq 1$ is not shift-invariant since it sends both $11$ and $22$ to $\{1,2\}$. 

Another simple case of a procedure that is not invariant is the (bilateral, memoryless) procedure \highlight{$\Spots^{evenodd}$} defined as follows: if the  desired spot $a_i$ is occupied, park ot the right if $a_i$ is even, and left if $a_i$ is odd.
\end{example}

\subsubsection{Local decision} If $I$ is any subset of $\Spots(W)$, let \highlight{$W_{|I}$} be the subword of $W$ given by the letters $a_i$ such that the $i$th car parked in $I$. We will only use this in the case where $I$ is a block of $S$.

\begin{definition}[\highlight{Local decision}]
\label{definition:local_decision}
A parking procedure is called locally decided if for any word $W$, letter $a$ and $I$  block of $\Spots(W)$ such that $a\in I$, we have that $\ls_\Spots(Wa)=\ls_\Spots(W_{|I}a)$.
\end{definition}

This can be interpreted roughly as the choice of a parking spot only depends on the preferred spots of the drivers who parked in the block where one wants to park. All examples of procedures defined until now are locally decided. It is easy to define one that is not, and we will an example in Section~\ref{sub:enumeration}.

\subsection{Local procedures} We now combine the two previous properties to obtain a class of particular interest.

\begin{definition}[\highlight{Local parking procedures}]
\label{definition:local}
A parking procedure $\Spots$ is \emph{local} if it is both shift-invariant and locally decided. 
\end{definition}

Let us give two immediate consequences of the definition. First, the next proposition shows that in order to study words resulting in a fixed set of parking spots $S$, it suffices to know parking functions.

\begin{proposition}
\label{prop:local_general_to_parking}
 Let $\Spots$ be a bilateral, local procedure. Let $S\in \fin(\bZ)$, with blocks $I_1,\ldots,I_m$ with $I_j=\{t_j+1,\ldots,t_j+r_j\}$ for $j=1,\ldots,m$.
 
 Then a word $W$ satisfies $\Spots(W)=S$ if and only if it is a shuffle of words $W_1,\ldots,W_m$ where $\tau^{-t_j}W_j$ is $\Spots$-parking of length $r_j$.

It follows that the total number of such words $W$ is 
\[\#\{W~|~\Spots(W)=S\}=\binom{r}{r_1,\ldots,r_m}\#\Park_{r_1}(\Spots)\cdots\#\Park_{r_m}(\Spots).\]
\end{proposition}

\begin{proof}
    Let $W_j$ be the subword of $W$ formed by cars that parked in $I_j$. It follows readily from Definition~\ref{definition:local_decision} that $\Spots(W_j)=I_j$. By Definition~\ref{definition:shift_invariance}, $\tau^{-t_j}W_j$ is $\Spots$-parking. Conversely, any shuffle $V$ of words $V_i$ such that $\Spots(V_j)=I_j$ will satisfy $\Spots(V)=S$. The formula follows immediately.
\end{proof}

As $\#\Park_{r}(\Spots)=(r+1)^{r-1}$ by Theorem~\ref{theorem:universal}, proved in the next section, we get a product formula for the number $\#\{W~|~\Spots(W)=S\}$. 

A second consequence is that in the \emph{memoryless} case, the class of local procedures is particularly easy to describe.

\begin{proposition}
\label{proposition:localmemoryless}
A local, memoryless, bilateral parking procedure $\Spots$ is determined by the data of $m_\Spots(r,i)\coloneqq M_\Spots(\{1,\ldots,r\},i)\in\{0,r+1\}$ for all $1\leq i\leq r$, or equivalently of the directions $\dir_\Spots(r,i)\coloneqq \dir_\Spots(\{1,\ldots,r\},i)$.

 Conversely, any choice of values $D(r,i)\in\{left,right\}$ for all $1\leq i\leq r$ determines a local, memoryless, bilateral parking procedure $\Spots$.% such that $\dir_\Spots=D$.
\end{proposition}

The proof follows readily from the definitions. Here are the explicit values $\dir_\Spots(r,i)$ in the memoryless examples encountered so far:
\begin{itemize}
\item For the usual procedure $\Spots^{right}$, we have $\dir(r,i)=right$ for any $r,i$.
\item For $\Spots^{closest}$, we have $\dir(r,i)=left$ if $i\leq r/2$ and  $\dir(r,i)=right$ if $i>r/2$. 
\item For $\Spots^{prime}$, we have $\dir(r,i)=right$ if $r$ is prime and $\dir(r,i)=left$ if $r$ is composite.
\end{itemize}

%%%%%%%%%%%%%%%%%%%%%%%%%%%%%
\subsection{Enumeration}
\label{sub:enumeration}

We will now prove Theorem~\ref{theorem:universal}, which says that for any bilateral local procedure $\Spots$, the number of $\Spots$-parking functions of length $r$ is given by $(r+1)^{r-1}$; the first values for $r=1,2,3,4$ are $1,3,16,125$. \smallskip

The proof is based on an argument of Pollak in the classical case, as found in~\cite{riordan1969ballots,foata1974mappings}. Let $\bZ^{i}_{[r+1]}$ be the set of words of length $i\leq r$ and with letters in $\{1,\ldots,r+1\}$. By identifying $\{1,\ldots,r+1\}$ with  $\bZ/(r+1)\bZ$, the cyclic group $\bZ/(r+1)\bZ$ then acts on $\{1,\ldots,r+1\}$ by addition, and thus on $\bZ^{i}_{[r+1]}$ by acting diagonally. Explicitly, let \highlight{$\rho=\rho_r$} be defined as 
\[\rho:\bZ/(r+1)\bZ\to \bZ/(r+1)\bZ,\quad i\mapsto i+1.\]
 which we extend naturally to words by $\rho(a_1\cdots a_i)=\rho(a_1)\cdots \rho(a_i)$. Thus each orbit of words has the form  $\{W,\rho(W),\ldots,\rho^r(W)\}$; we call such orbits ``cyclic''. Equivalently, they correspond to the cosets of the subgroup generated by $(1,\ldots,1)$ in $\bZ^{i}_{[r+1]}$ (considered as the additive group $(\bZ/(r+1)\bZ)^i$).

\begin{proposition}
\label{prop:cyclic_lemma}
Let $\Spots$ be a bilateral, local procedure, and $r\geq 1$. 
There is exactly one $\Spots$-parking function in each cyclic orbit in $\bZ^{r}_{[r+1]}$.
\end{proposition}

\begin{proof}[Proof of Theorem~\ref{theorem:universal}]
It is clear that $\Park_r\subset \bZ^{r}_{[r+1]}$, the latter being of cardinal $(r+1)^r$. Proposition~\ref{prop:cyclic_lemma} then tells us that a fraction $1/(r+1)$ of these words are parking functions, which proves the desired enumeration.
\end{proof}

 For the proof of Proposition~\ref{prop:cyclic_lemma}, we will introduce a cyclic version of the procedure $\Spots$. By a \highlight{cyclic parking procedure}, we mean a procedure with domain $\bZ^{\leq r}_{[r+1]}$ and codomain $\fin(\{1,\ldots,r+1\})$ that satisfies the properties of Definition~\ref{definition:parking_procedure}. Cyclic intervals are well-defined as sets of consecutive integers, so we can also define cyclic blocks as the maximal such intervals in a set $S$.  Bilateral cyclic parking procedures thus make sense, and can be specified by a direction function $\dir$.

\begin{proof}[Proof of Proposition~\ref{prop:cyclic_lemma}]
We define the cyclic procedure $\Spots_{r}$ by induction, by specifying its direction function $\dir_r$. Suppose that it is defined on all words of length up to $i<r$, and let $W$ have length $i$. Let $S\coloneqq\Spots_{r}(W)$, $a\in S$, and let $I$ be the cyclic block of $S$ containing $a$ and $V=W_{|I}$. By applying a suitable rotation $\rho^k$, we have $\rho^k(I)=\{1,\ldots,j\}$ for a certain $j$. 

We then define $\dir_{r}(Wa)\coloneqq\dir(\rho^k(Wa))$. Let us remark that at this point, the local property has not been used, and the definition makes thus sense for any procedure. Note that $\Spots$ being local implies that this is also equal to $\dir(\rho^k(Va))$. Now we claim that:
\begin{enumerate}
\item For any $W$, $\Spots_{r}(\rho(W))=\rho(\Spots_{r}(W))$. \label{claim1}
\item $W$ is a $\Spots$-parking word if and only if $\Spots_{r}(W)=\{1,\ldots,r\}$. \label{claim2}
\end{enumerate}

The two claims together imply the proposition: the first claim implies that each orbit contains exactly one word $W$ with $\Spots_{r}(W)=\{1,\ldots,r\}$, while the second one claims that these are precisely the $\Spots$-parking words. Let us now prove these claims.

Claim~\eqref{claim1}  follows readily from our definition of $\Spots_{r}$: keeping the notations from the above paragraph, one needs to show that $\dir_r(Wa)=\dir_r(\rho(Wa))$ for any $a,W$ with $a\in\Spots_{r}(W)$, and indeed both are given by $\dir(\rho^k(Wa))=\dir(\rho^{k-1}\rho(Wa))$. For Claim~\ref{claim2}, note that $r+1$ is missing from $\Spots_{r}(W)$ at the end if and only if it was missing at each step. If $W'$ is a $\Spots$-parking word, let us show that  $\Spots(W)=\Spots_{r}(W)$ for each prefix. We need to show that $\dir(Wa)=\dir_r(Wa)$ for each prefix $Wa$ of $W'$ with $a\in \Spots(W)$. Since $\Spots$ is local, $\dir(Wa)=\dir(Va)$  while $\dir_{r}(Wa)=\dir(\rho^k(Va))$. Now remark that we have $\rho^k(Va)=\tau^j(Va)$ for a certain $j$ as words in $\{1,\ldots,r\}$, and we can conclude by shift-invariance.\end{proof}

Neither condition in the definition of a local procedure (Shift invariance and local decision) can be removed as a hypothesis:\smallskip

$\bullet$ The procedures $\Spots^{evenodd}$ and $\Spots^{Nap,k}$ (for $k>0$) are locally decided, but not shift invariant. $\Spots^{evenodd}$ has only $2$ parking functions of length $2$, while $\Spots^{Nap,k}$  has $4$, so the result of the theorem does not hold for either of the procedures.\smallskip

$\bullet$
Now consider the following memoryless procedure \highlight{$\Spots^{far}$}: if one's preferred spot $a$ is occupied, let $R$, \emph{resp.} $L$, be the number of cars already parked to the right, \emph{resp.} to the left,  of $a$. Then one parks in the nearest spot available to the right if $R\leq L$, and to the left if $R> L$.
In our notations, let $a,W$ such that  $a\in S\coloneqq \Spots(W)$, then $R=\#\{i\in S~|~i>a\}$ and $L=\#\{i\in S~|~i<a\}$. Then define $\dir_{\Spots}=left$ if $L\geq R$ while $\dir_{\Spots}=right$ if $L<R$.

 It is shift invariant, but not locally decided. Direct enumeration shows that there are only $14$ parking functions of length $3$ for $\Spots^{far}$, so the conclusion of Theorem~\ref{theorem:universal} does not hold\footnote{Proposition~\ref{prop:cyclic_lemma} fails for $r=3$, as two cyclic orbits have no parking function ($\{131,242,313,424\}$ and $\{133,244,311,422\}$) while the other fourteen have one.}

\subsection{Extending Theorem~\ref{theorem:universal} and Proposition~\ref{prop:cyclic_lemma}}
\label{sub:weakening_hypotheses}
 Following up on the previous exmaple, the problem is essentially that the parking decisions in $\Spots^{far}$ depend on information on the position (left or right) of cars that are outside of the block where one wants to park. Such information is not stable under a cyclic rotation of the blocks, which is why the cyclic version $\Spots^{far}_r$ used in the proof of Proposition~\ref{prop:cyclic_lemma} behaves badly.

For a bilateral, shift-invariant procedure, the local decision property removes this issue. As it is also easy to state and has some nice structural properties (Propositions~\ref{prop:local_general_to_parking} and~\ref{proposition:localmemoryless}), we chose to state the theorem at this level of generality. But there are weaker properties that also work, the key being to be able to define the cyclic version. Analyzing the proof of Proposition~\ref{prop:cyclic_lemma}, one can weaken the hypothesis as follows.

\begin{proposition}
\label{prop:cyclic_lemma_extended}
Let $\Spots$ be a bilateral procedure that satisfies the following property: 
For any $r\geq 1$, any word $W$ of length at most $r$ with letters in $\{1,\ldots,r\}$, and any cyclic rotation $R=\rho_r^k$ such that $\Spots(R(W))$ does not contain $r+1$, then $\dir(Wa)=\dir(R(Wa))$ for any $a\in\Spots(W)$.\\
Then there is exactly one $\Spots$-parking function in each cyclic orbit in $\bZ^{r}_{[r+1]}$, and thus there are $(r+1)^{r-1}$ $\Spots$-parking words of length $r$.
\end{proposition}

In particular, note that the property is satisfied if the parking rule is a function of the total number of parked cars --equivalently, if the parking rule for the $i$th car depends on $i$. For this class of procedures, the previous result is in fact hinted at in the seminal article \cite{konheim1966occupancy}, as we will precise in  the next section.

%%%%%%%%%%%%%%%%%%%%%%%%%%%%%
\section{Probabilistic parking}
\label{sec:probabilistic}
%%%%%%%%%%%%%%%%%%%%%%%%%%%%%

To add randomness to the setting, one way is to study properties of $\Spots$-parking words picked uniformly at random for instance. In the classical case, this is a very natural problem from the hashing viewpoint, and the problem  and was studied extensively: see~\cite{diaconis2017probabilizing,flajolet1998hashing}. One can also randomize the procedures themselves, which is the point of view we develop here. 

\subsection{Setting} A procedure  which we will denote \highlight{$\Spots^{KW,q}$} is defined at the end of the seminal article~\cite{konheim1966occupancy}: Fix $q\in [0,1]$. When one's spot is occupied, park at the nearest spot to the right ( \emph{resp.} left) with probability $q$ (\emph{resp.} $1-q$). A more generalized procedure \highlight{$\Spots^{KW,(q)_i}$} is in fact defined at the very end of the paper: if the $i$th driver's spot is occupied then the probabilities are $q_i$ and $1-q_i$ for a fixed family $\mathbf{q}=(q_i)_{i\geq 1}$. Notice that if each $q_i$ is in $\{0,1\}$, we get a class of deterministic procedures that satisfies the hypothesis of Proposition~\ref{prop:cyclic_lemma_extended}.\smallskip

We can develop the notion of a bilateral probabilistic parking procedure $\Spots$ in much the same way as we did in the deterministic case. In that case $\Spots(W)$ will the data of a finitely supported ``probability measure'' on $\fin(\bZ)$; we write it $\Spots_{W}:\fin(\bZ)\to[0,1]$. It is zero outside of a finite number of subsets, and its values sum to $1$.

To define rigorously a bilateral probabilistic procedure $\Spots$, one can take a combinatorial approach.  Fix \highlight{$\dir_{\Spots}(W,S,a)$} real numbers in $[0,1]$, representing the probability to go right when $a\in S$  after having read $\Spots_{W}$. These are defined when $W$ has $\#S$ letters, all of which belonging to $S$. Let us define the (directed, acyclic) graph $G(\Spots)$ with vertices all such pairs $(W,S)$, with outgoing edges at each vertex labeled by $\bZ$. Then we have the following edges for any $W,S,a$:
\begin{itemize}
\item if $a\notin S$, $(W,S)\stackrel{a}{\to} (Wa,S\sqcup\{a\}$ with weight $1$. 
\item if $a\in S$, let $I=\{s,\ldots,t\}$ the block that contains $a$. Then we have an edge $(W,S)\stackrel{a}{\to} (Wa,S\cup\{t+1\}$ with weight $\dir_{\Spots}(W,S,a)$ and an edge $(W,S)\stackrel{a}{\to} (Wa,S\cup\{s-1\}$ with weight $1-\dir_{\Spots}(W,S,a)$.
\end{itemize} 

The value $\Spots_{W}(S)$ is then the total weight of all paths in $G(\Spots)$ going from $(\epsilon,\emptyset)$ to $(W,S)$. It is clear by induction that the sum of $\Spots_{W}(S)$ over $S\in \fin(\bZ)$ is $1$, which justifies thinking of $\Spots_{W}$ as a probability measure.

\begin{definition}[Parking probability]  Let $\Spots$ be a probabilistic parking procedure, and $W$ a word of length $r$. Then the \highlight{parking probability $\mathrm{Park}_{\Spots}(W)$} is defined as $\Spots_{W}(\{1,\ldots,r\})$. 
\end{definition}

The notion of a \textit{memoryless} procedure extends directly: here we need just fix 
real numbers $\dir_{\Spots}(S,a)$ in $[0,1]$. The graph $G(\Spots)$ has correspondingly vertices indexed simply by $\fin(S)$ and its edges are simply defined by dropping $W$ from the definitions above. Finally $\Spots_{W}(S)$ is the total weight of all paths in $G(\Spots)$ from $\emptyset$ to $S$ and labeled by the word $W$.

\subsection{Enumerative properties}
\label{sub:proba_enum}

 The notion of local procedure extends readily to the probabilistic setting.
 
 We have the following extension of Theorem~\ref{theorem:universal}.

\begin{theorem}
Let $\Spots$ be a local, probabilistic parking procedure, and $r\geq 1$. Then the sum of $P_\Spots(W)$ over all words $W$ of length $r$ is $(r+1)^{r-1}$.
\end{theorem}

This was noticed to hold for the procedure $\Spots^{KW,q}$ in \cite{konheim1966occupancy}. The key to the theorem is the following extension of Proposition~\ref{prop:cyclic_lemma}:

\begin{proposition}
Let $\Spots$ be a local, probabilistic parking procedure, and $r\geq 1$. Then the  sum of $P_\Spots(W)$ over any cyclic orbit in $\bZ^{r}_{[r+1]}$ is equal to $1$.  Thus the sum of $P_\Spots(W)$ over all words of length $r$ is $(r+1)^{r-1}$  
\end{proposition}

We skip the proofs as they are direct adaptations of the previous ones.\smallskip

The local hypothesis can be weakened to the one from Proposition~\ref{prop:cyclic_lemma_extended}. In particular, the result applies to the procedures $\Spots^{KW,(q)_i}$, and thus we recover as a special case the result mentioned at the very end of \cite{konheim1966occupancy}.

\subsection{Abelian procedures} 
\label{sub:abelian}

A local, memoryless probabilistic procedure $\Spots$ is completely specified by the numbers $p_\Spots(r,i)=\dir_\Spots(r,i)\in [0,1]$ giving the probability to go right when a car with preferred spot $i$ wants to park on the block $\{1,\ldots,r\}$. 

\begin{example}
For the procedure $\Spots^{KW,q}$, we have $p(r,i)=q$ for all $r,i$.
\end{example}

\begin{example} Let us define the procedure \highlight{$\Spots^{q}$}, with $q\in [0,+\infty]$, by $p_{\Spots^{q}}(r,i)=\frac{[i]}{[r+1]}$ where $[j]$ denotes the $q$-integer $1+j+\cdots +q^{j-1}$: this is the classical probability that the car parks in $r+1$ after a biased random walk on $\{1,\ldots,r\}$ where the probability to go right is $1/(1+q)$. 

This leads to the rich combinatorics of {\em remixed Eulerian numbers}, introduced by the author with Vasu Tewari~ \cite{nt2021qKlyachko,nt2023remixed}. The interpretation as a parking procedure is a direct reformulation of the ``sequential process'' given in \cite[Section 2]{nt2023remixed}.
\end{example}

\begin{definition}
A (probabilistic) parking procedure $\Spots$ is called \emph{abelian} if 
 $\Spots_{a_1\cdots a_r}=\Spots_{a_{\sigma_1}\cdots a_{\sigma_r}}$ for any letters $a_1,\ldots,a_r$ and any permutation $\sigma$.
\end{definition}

This means that the outcome of the procedure is independent of the orders in which the cars arrive. $\Spots^{right}$ is known to be abelian, as can be seen immediately from the characterization of parking words in the introduction. More generally, the procedures $\Spots^{q}$ are known to be abelian, cf. \cite{nt2023remixed}.

We have conversely the following uniqueness result:

\begin{proposition}
\label{prop:memoryless_abelian}
The procedures $\Spots^{q}$, $q\in [0,\infty]$, are the only local memoryless procedures that are abelian.
\end{proposition} 
 
In the deterministic case, it follows that a local memoryless procedure $\Spots$ is abelian if and only if it is either $\Spots^{right}$ or its opposite version $\Spots^{left}$.

\begin{proof}
Let $\Spots$ be a local memoryless abelian procedure, and write $p_{r,i}\coloneqq p_\Spots(r,i)$ for $1\leq i\leq r$. Define $q$ by $p_{1,1}=\frac{1}{1+q}$. We need to show that $p_{r,i}=\frac{[i]}{[r+1]}$ for all $r,i$. We prove this by induction on $r$. It holds for $r=1$ by definition, and we assume the claim holds for $r-1$ with $r\geq 2$.

Consider the word $W_{r,i}=12\cdots (r-1)ri$ for any $1\leq i\leq r$. By definition its parking probability is $p_{r,i}$. Now, for $i\neq r$, consider the word $12\ldots (r-1)ir$, obtained by exchanging the last two letters in $W_{r,i}$. Its parking probability is given by $p_{r-1,i}p_{r,r}$, so by abelianity and the induction hypothesis we obtain:
\begin{equation}
\label{eq:firsteq}
p_{r,i}=\frac{[i]}{[r]}p_{r,r}\text{ for }i=1,\ldots,r-1.
\end{equation}
Consider now the word $12\cdots (r-2)rr(r-1)$, obtained from $W_{r,r}$ by moving $n-1$ to the end. Its parking probability is easily computed as $\frac{1}{1+q}+\frac{q}{1+q}p_{r,r-1}$, the two terms corresponding to the penultimate car going right or left. By abelianity.
\begin{equation}
\label{eq:secondeq}
p_{r,r}=\frac{1}{1+q}+\frac{q}{1+q}p_{r,r-1}.
\end{equation}

One can now easily solve the system of equations~\eqref{eq:firsteq},\eqref{eq:secondeq} and get the desired result $p_{r,i}=\frac{[i]}{[r+1]}=p_{\Spots^{q}}(r,i)$. 
\end{proof}

%%%%%%%%%%%%%%%%%%%%%%%%%%%%%
\section{Encoding with binary trees}
\label{sec:binary_trees}
%%%%%%%%%%%%%%%%%%%%%%%%%%%%%

In this section we will define a general {lift} of the parking procedure: given any parking procedure $\Spots$, we define an injective function $\Tspots$ also defined on $\bZ^*$ with a natural projection $\Pi$ such that $\Pi\circ \Tspots= \Spots$.

\subsection[Injective lift]{Definition of $\Tspots$}

Recall that a finite, plane, binary \emph{tree} is defined recursively as either empty or consisting of a node, a left (sub)tree and a right (sub)tree. Its size is its number of nodes. These are in bijection with \emph{complete} binary trees, where by attaching extra leaves. A forest is then usually defined as a set of trees. 

\begin{definition}
An \emph{indexed forest} $F$ is the data of a subset $S$ in $\bZ$, and a binary tree of size $|I|$ for each block of $I$ of $S$. The set $S$ is the {\em support} of $F$. Binary trees of size $r$ are identified with indexed forests with support $\{1,\ldots,r\}$.
\end{definition}

 In the example below, the support is $\{2,3,4,5,6,7,11,14,15\}$. Elements of the support correspond bijectively to nodes of the forest, as illustrated by the arrows in the figure: this is the \highlight{canonical labeling} of the nodes that will be important in what follows, which we will use to specify nodes.

 \begin{remark}
     This corresponds to the notion introduced in~\cite{nt2024forest} by V. Tewari and the author. It was then recovered naturally in~\cite{NST_a} in further collaboration with H. Spink. In this last work, indexed forests are equivalently defined as sequences of trees, and one should beware that the notion of support is slightly different. 
 \end{remark}

\begin{figure}[!ht]
\centering
\includegraphics[width=0.8\textwidth]{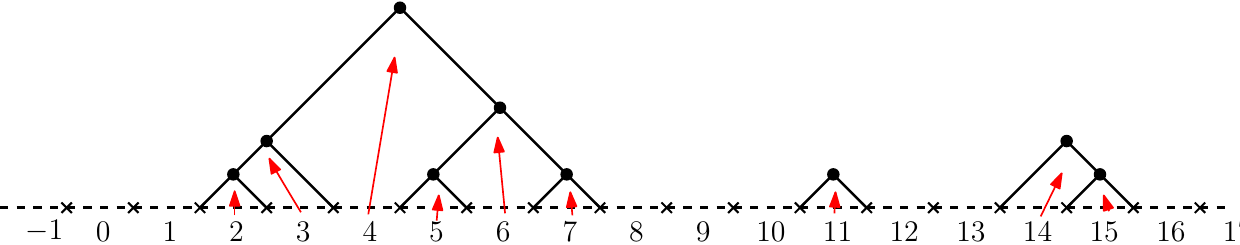}
\end{figure}

Fix any bilateral procedure $\Spots$. Given any word $W$, we will attach to it a pair $(P,Q)$ of labeled forests with the same underlying indexed forest.
The projection \highlight{$\Pi(P,Q)$} is defined as the support of this indexed forest $F$. The labels are attached to the nodes of the common shape, and we will write $p_i$, \emph{resp.} $q_i$ for the label in $P$, \emph{resp.} $Q$ of the node $i\in\Pi(P,Q)$. $P$ will be bijectively labeled  by the multiset of letters of $W$, while $Q$ is a \emph{decreasing} forest: it has labels $\{1,\ldots,r\}$ and each node has greater label than all its descendants.\medskip

To give the precise construction, let us first focus on the forest $Q$. Consider the injection $\pi_\Spots(W)$ defined in Section~\ref{sec:bilateral}, which encodes the order in which parking spots were filled. Let $S=\Spots(W)$. Then one can naturally encode $\pi_\Spots(W)$ as an indexed decreasing forest with support $S$: the standard bijection between permutations and labeled trees naturally extends. This defines the forest $Q$, and its underlying forest $F$. We note $i\in S\to q_i$ this labeling. To construct $P$, simply label $F$ by having node $i$ get the label $p_i=W_{q_i}$.

\begin{definition}
We define $\Tspots: \bZ^*\mapsto (P,Q)$ to be this construction, which we call the $\Spots$-correspondence attached to the parking procedure $\Spots$.
\end{definition}

Here is an illustration for $\Spots^{LBS}$ and the word $W=5.11.8.3.9.3.2$:

\begin{figure}[!ht]
\centering
\includegraphics[width=0.7\textwidth]{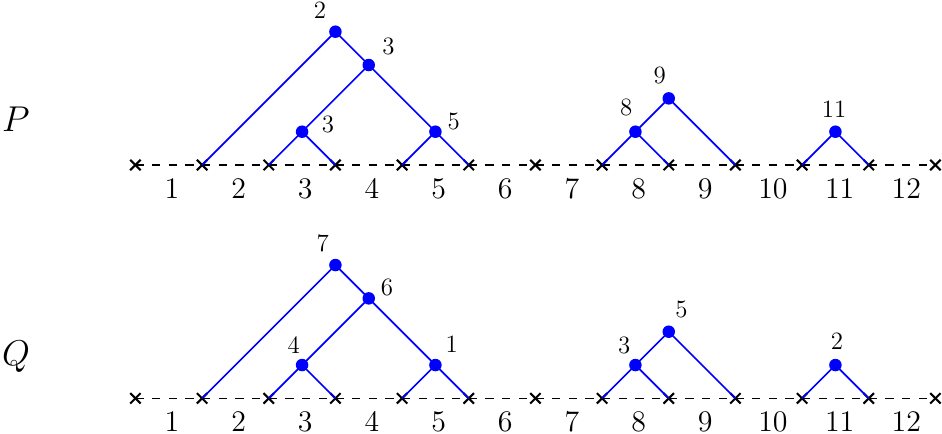}
\end{figure}

 Let us list some immediate properties. 

\begin{proposition}
Let $\Spots$ be a bilateral parking procedure:
\begin{itemize}
\itemsep0pt
\item A word $W$ is $\Spots$-parking if the common shape of $(P,Q)$ is a tree with support of the form $\{1,2,\ldots,r\}$. 
\item The $\Spots$-correspondence is injective.
\item The canonical labeling of a node is equal to the spot where the corresponding car ended up parking.
\end{itemize}
\end{proposition}

Following up on this last point, one sees that $\sum_i|p_i-i|$ corresponds to the 
\highlight{total displacement} of the parking process. 

%%%%%%%%%%%%%%%%%%%
\subsection{Correspondence}
\label{sub:correspondence}

The framework is particularly nice if the following property holds: for any $(P,Q)$ in the image of $\Tspots$, all $(P,Q')$ with $Q'$ a decreasing forest of the same shape is also in the image of $\Tspots$. In that case the image of $\Tspots$ is determined once we know what labeled forests $P$ can occur. We will say that $\Tspots$ is a \highlight{good correspondence}.

In general, local procedures do not give rise to good correspondences. It holds however in the memoryless case:

\begin{proposition}
Let $\Spots$ be a locally decided, memoryless procedure. Then $\Tspots$ is a good correspondence. The corresponding forests $P$ are given by the following condition: if $i$ is any node with interval $[l_i,r_i]$, then the possible values of its label come from the following set :
\begin{equation}
\{j~|~\dir_\Spots(\{l_i,\ldots,i-1\},j)=right\}\cup
 \{i\}\cup \{j~|~\dir_\Spots(\{i+1,\ldots,r_i\},j)=left\}.
\end{equation}
\end{proposition}

\begin{proof} Suppose $W\mapsto(P,Q)$. Exchange $q_i=k$ and $q_j=k+1$ in $Q$ when $i,j$ are incomparable, let $Q'$ be the result. Then we claim that $W'\mapsto Q'$ where $W'$ is obtained from $W$ by swapping letters at positions $k,k+1$. Indeed this is true for up to position $k$ because it's locally decided, and then the memoryless condition takes care of the rest. 

Since all decreasing forests can be obtained in this way, this shows that $\Tspots$. The rest of the statement is easy and left to the reader.
\end{proof}

 To determine the fibers of the outcome function $\pi_\Spots$, one has to simply count the number of $\Spots$-trees of a given shape. We thus get immediately from the previous proposition:

\begin{proposition}
\label{proposition:outcome}
Let $\Spots$ be a locally decided memoryless procedure, and $\sigma$ a permutation of $\{1,\ldots,r\}$. Let $T$ be the binary tree underlying the decreasing tree of $\sigma$. 

Then the number of $\Spots$-parking functions with outcome $\sigma$ is given by the product over all $i$ of the cardinality of the sets in Proposition~\ref{proposition:outcome}. If $\Spots$ is moreover shift-invariant (thus a local procedure), then we have explicitly
\begin{equation}
\label{eq:fiber}
\prod_{i\in T} (1+R_{i-\ell_i}+L_{r_i-i}),
\end{equation}
where $R_k$ is the number of values $i\in\{1,\ldots,k\}$ such that $\dir_\Spots(r,i)=right$, and $L_k=k-R_k$. 
\end{proposition}

The $\Spots^{right}$-forests for the classical procedure $\Spots^{right}$ have labels in $[l_i,i]$ for any node $i$. The specialization of Proposition~\ref{proposition:outcome} then gives the well-known result given in~\cite[Exercise 5.49(d)]{stanley1999ec2}.% and reproven in~\cite{colmenarejo2021countingknaples}

\begin{example} For the $k$-Naples procedure $\Spots^{Nap,k}$ the set in Proposition~\ref{proposition:outcome} 
is given by
\[\{l_i+k-1,l_i+k,\ldots,i-1\}\cup\{i\}\cup\{i+1,\ldots,i+k\}\]
 if $l_i>1$, and by
 \[\{1,\ldots,i-1\}\cup\{i\}\cup\{i+1,\ldots,i+k\}\]
  if $l_i=1$.

We thus recover the main result Theorem 3.12 of \cite{colmenarejo2021countingknaples}. Note that $l_i=1$ means that node $i$ lies on the leftmost branch of the tree $T$. It is then immediate to get the recursive formula for $\#\Park_r(\Spots^{Nap,k})$ given in~\cite{christensen2020naples}.
\end{example}

To finish this section, let us specialize to the classical case. We get a bijection $\Phi=\Tspots^{right}$ between usual parking functions and binary trees with a double labeling.

Parking functions are known to be in bijection with families of labeled binary trees in (at least) two other ways. The first one uses the standard representation of parking functions as labeled Dyck paths, composed with any bijection between Dyck paths and binary trees. This was done recently for example in ~\cite[Section 6]{bergeron2021hopfparking}. Another known bijection is with \emph{Shi trees}, as defined by Gessel, see for instance~\cite{bernardi2018deformations} and references therein. 

These bijections are however not naturally related to our correspondence $\Phi$, in the sense that the number of labelings of a given tree differ for $r=3$: If we write the number of labelings in a nonincreasing fashion for all $5$ trees of these sizes, we get $6,4,3,2,1$ in our correspondence, $6,3,3,3,1$ using any of the other bijections.
 
In fact, using some variation of the standard bijection between binary and plane trees, $\Phi$ can be seen to be equivalent to a bijection between parking functions and Cayley trees due to Knuth, as described for instance by Yan~\cite{yan2015parking}. 

\subsection{The probabilistic case}
Constructions generalize naturally: each path in the graph $G(\Spots)$ naturally gives rise to a pair $(P,Q)$ by using the above construction, carrying naturally the weight of the path. If we group all pairs $(P,Q)$ ending in $(W,S)$, we get an interpretation for $\Spots_W(S)$ as the total weight of such pairs. In the case of $\Spots^{q}$, this gives rise naturally to the interpretation of $A_c(q)$ as counting {\em Postnikov trees}~\cite{nt2023remixed}.
 
%%%%%%%%%%%%%%%%%%%%%%%%%%%%%%
\section{Colored version}
\label{sub:extra}
%%%%%%%%%%%%%%%%%%%%%%%%%%%%%%

Let $C$ be any set, and consider the alphabet $A=\bZ\times C$. The set $C$ represents some extra information: in terms of cars, one might consider its brand, its color, or the age of the driver. Let $\val:A=\bZ\times S \to \bZ$ be the projection to the first factor, which represents the preferred spot. Now a preference list consists of a word $a_1\ldots a_r$ in $A^*$, the $i$th car having preferred spot $\val{a_i}$. The set $C$ can then be used as extra source of information in order to decide where to park.

We have as before (bilateral) parking procedures,
 \[\Spots:A^*\to \fin(\bZ)\]
  defined exactly as in Section~\ref{sec:bilateral}: one simply has to replace $\bZ^*$ by $A^*$ and the conditions $a\in \Spots(W)$  by $\val(a)\in \Spots(W)$. %Note that the notion of memoryless can be extended also but is not useful since the set $S$ then plays no role

If $a=(i,s)$, extend the shift $\tau$ by defining $\tau(a)=(i+1,s)$ and extend to words. Then the shift-invariance and local decision property are immediately extended to procedures on $A^*$, so we have the notion of local procedures in this case. 

One can require $\Spots$ to be only a \emph{partial} function, that is, to be defined on a subset $L\subset A^*$. It is reasonable to require that $L$ be closed under deleting a letter at any position, that is, that {$L$ be closed under taking subwords}. This ensures that the words $W_{|I}$ are well-defined in the definition of a local procedure. Also, $L$ should be closed under (cyclic) shifts: if $W$ is a word of $L$ of length $i<r$, the word $\rho_r(W)$ is also in $L$. We can now state an extension of Proposition~\ref{prop:cyclic_lemma}.

\begin{proposition} Let $L\subseteq A^*$ be closed under subwords and cyclic shifts. Let $\Spots$ be a colored local procedure defined on $L$.  Then any cyclic class
in $L\cap A^{r}_{r+1}$ contains exactly one parking word.
\end{proposition}

We can extend $\Spots^{LBS}$ to this colored setting, to get the procedure introduced and studied in~\cite{nt2024forest}. Pick $C=\bZ_{>0}$, and consider the language $L$ of words in $A=\bZ\times C$ with \emph{distinct letters}.
 Order $A$ lexicographically. Let $W,a,I$ such that $I$ is a block of $\Spots^{LBS}(W)$ and $a\in I$. Let $j\in\{1,\ldots,r\}$ be maximal such that $a_j\in I$ (``last car that parked on the interval''). Then if $a<a_j$ the cars parks left, and if $a>a_j$ park right. Since we only consider words with distinct letters, equality can not occur.

\bibliographystyle{plain}
\bibliography{Parking_biblio}

\end{document}